\documentclass{article}

\usepackage{amssymb, amsmath, amsfonts, epsfig, color,amscd,amsthm,setspace,lineno}

\newtheorem{theorem}{Theorem}[section]

\newtheorem{conjecture}[theorem]{Conjecture}
\newtheorem{corollary}[theorem]{Corollary}

\newtheorem{proposition}[theorem]{Proposition}

\theoremstyle{definition}

\newtheorem{definition}[theorem]{Definition}

\begin{document}

\title{Omnipersistent Signatures}
\author{James W Anderson\\School of Mathematical Sciences\\University of Southampton\\Southampton SO17 1BJ England\\j.w.anderson@southampton.ac.uk\\ \\Aaron Wootton\\Department of Mathematics\\University of Portland\\Portland OR 97203 USA\\wootton@up.edu}

\maketitle

\begin{abstract}
In this note, we lay the groundwork for a new approach to the problem of group-signature classification of group actions on closed Riemann surfaces. This new approach first focuses on analyzing the low level arithmetic conditions on signatures before invoking the more complicated group theory. We provide the complete first step in this approach by giving the complete list of signatures which arithmetically {\em could} appear as a signature in every possible genus, and the subset of those which {\em do} appear as the signature of a group action in every possible genus. 

\medskip
\noindent
MSC 14H37; 30F20; 57M60

\medskip
\noindent
keywords: Riemann surface, automorphism group, signature
\end{abstract}

\section{Introduction}
\label{introduction}

The purpose of this note is two-fold.  First, we determine the complete list ${\mathfrak p} = \cap_{\sigma\ge 2} {\mathcal P}_\sigma$ of omnipersistent potential signatures, by proving that there exists a lattice structure on the space ${\mathcal P} = \{ {\mathcal P}_{\sigma}\: |\: \sigma\ge 2\}$ of potential signatures.  

In words, the tuple $(h; m_1,\ldots, m_r)$ for $h\ge 0$ and $r\ge 0$ is a {\em potential signature} in genus $\sigma\ge 2$, and so by definition lies in ${\mathcal P}_{\sigma}$, if it satisfies the number-theoretic conditions of being a signature in that genus, as given in Definition \ref{definition potential signature}.  (In the case $r=0$, we have the $1$-tuple $(h; -)$.) We do not require that this tuple be the signature of the action of some group acting by conformal homeomorphisms on some closed Riemann surface of genus $\sigma$.   A potential signature is {\em omnipersistent} if it arises as a potential signature for every genus $\sigma\ge 2$.  

In this note, we prove the following.

\medskip
\noindent
{\bf Theorem} {\bf \ref{theorem potential omnipersistent}.} {\em The omnipersistent potential signatures are 
\[ {\mathfrak p} = \cap_{\sigma\ge 2} {\mathcal P}_\sigma = {\mathcal P}_2. \] }

Second, we determine the complete list ${\mathfrak a} = \cap_{\sigma\ge 2} {\mathcal A}_\sigma$ of omnipersistent actual signatures.  In words, the tuple $(h; m_1,\ldots, m_r)$ for $h\ge 0$ and $r\ge 0$ is a signature in genus $\sigma$ (with the same convention as above, that in the case $r=0$ we have the $1$-tuple $(h; -)$), and so by definition lies in ${\mathcal A}_\sigma$, if there exists a closed Riemann surface $X$ of genus $\sigma\ge 2$ and a group $G$ acting by conformal homeomorphisms on $X$ for which the quotient $X/G$ has genus $h$ and $r$ branch points of orders $m_1,\ldots, m_r$.  (And so the case $r=0$ arises for a fixed point free action of $G$ on $X$.) A signature is {\em omnipersistent} if it arises as a signature for every genus $\sigma\ge 2$.

In this note, we prove the following.

\medskip
\noindent
{\bf Theorem} {\bf \ref{theorem omnipersistent}.} {\em The omnipersistent actual signatures are 
\[ {\mathfrak a} = \cap_{\sigma\ge 2} {\mathcal A}_\sigma = \{ (2; -), (1;2,2), (0;2,2,2,2,2), (0;2,2,2,2,2,2)\}. \] }
\medskip

Though interesting in its own right, our work here is part of a larger plan: we wish to understand possible structures on the space of signatures of actions of groups of conformal homeomorphisms on closed Riemann surfaces, and our work here provides the groundwork for a new approach to this problem.  

It is a classical undertaking in the theory of closed Riemann surfaces to determine the complete list of all finite groups that act on a Riemann surface of some genus $\sigma\ge 2$ and the signatures of these actions.  

One standard approach to this problem is to fix a genus $\sigma \ge 2$ and invoke a brute force search for groups acting by conformal homeomorphisms on some closed Riemann surface of genus $\sigma$ and the signatures arising from these actions.  As there are only finitely many groups and signatures which can occur in a fixed genus, each can be tested with respect to the necessary and sufficient conditions provided by Riemann's existence theorem, see Theorem \ref{theorem existence}, to check which ones occur. 

Increased computational power has provided significant progress using this approach over the last few decades, starting with complete lists for surfaces of genus $3$ by Broughton \cite{Bro}, followed by other results for other small genera.  An inexhaustive list includes for example Bogopol'ski\u\i \: \cite{Bog}, Kuribayashi and Kimura \cite{Kur} and  Breuer \cite{Breu}. 

A second approach to the problem is to instead fix a group, or family of groups, and determine all actions by conformal homeomorphisms for that family over all closed Riemann surfaces  by allowing the genus $\sigma$ to vary. Rather than filling out a complete list for a fixed genus, this approach focuses on adding actions to infinite sequences of genera, the general philosophy being that as more such families are added, the lists for each individual $\sigma$ will become closer to the actual complete list. 

We do not provide a complete list here, but there are many examples illustrating this approach. For example, in Kulkarni \cite{Kul}, it is shown that for every genus $\sigma\ge 2$, there is a cyclic group action of order $4\sigma+2$ acting on a closed Riemann surface of genus $\sigma$ with signature $(0;2,2\sigma+1,4\sigma+2)$, and in Breuer \cite{Breu}, it is shown that for every genus $\sigma\ge 2$, there is an Abelian group action by a group of order $4\sigma+4$ on a closed Riemann surface of genus $\sigma$ with signature $(0;2,2\sigma+2,2\sigma+2)$.

There are many ways to refine each of these searches that will speed up progress, but the number of possible groups and their complexity significantly increase as genus increases, so a definitive answer for arbitrary $\sigma$ will always remain to be elusive -- there will never be enough computational power to provide a definitive answer. In short, the group theoretical conditions for a signature to be the signature of an action are computationally expensive to apply, for large genus and for groups of large order.

A second disadvantage to these approaches is that they do not provide insight into some of the basic phenomena we see within the lists themselves. For example, there are many signatures, such as $(0;2,3,7)$, which make regular appearances, whereas there are other signatures which rarely show up. As another example, contrary to naive expectation, the number of distinct group actions does not grow in a predictable way with genus; for example, in genus $17$ there are $733$ distinct actions but just $337$ for genus $18$.     

In light of the difficulties with these two classical approaches, we propose here a third approach to the problem of group-signature classification: first focus on analyzing spaces of potential signatures and how they vary as we allow the genus $\sigma$ to vary, and only afterwards invoke the group theory. 

The general philosophy for this approach is as follows: since the conditions for when a tuple is a potential signature in a given genus are much easier to check than those for whether a tuple is an actual signature in that genus, we can obtain rather complete information about potential signatures and how they relate to each other as the genus varies.  With these lists known, we can start to add families of group actions through varying genera by invoking the necessary group theory, the lists for each individual $\sigma$ becoming closer to the actual complete list the more families we add. 
 
The interest of Corollary \ref{cor-lat} and Theorems \ref{theorem potential omnipersistent} and \ref{theorem omnipersistent} are that they provide the first natural step in this new approach. Specifically, Corollary \ref{cor-lat} gives the general structure of how potential signatures for group actions vary between genera. Theorem \ref{theorem potential omnipersistent} then provides the complete list of tuples which {\em could} appear as a signature in every possible genus, and Theorem \ref{theorem omnipersistent} gives the subset of those which {\em do} appear as the signature of a group action  in every possible genus. 

\section{Preliminaries}
\label{preliminaries}

In this Section, we introduce the necessary preliminaries.  We start with our foundational definition.

\begin{definition} (See for example Proposition 2.1 of Broughton \cite{Bro}) By a {\em potential signature} for genus $\sigma\ge 2$, we mean a tuple $(h; m_1,\ldots, m_r)$ of non-negative integers for $h\ge 0$ and $r\ge 0$ that satisfies the two necessary number-theoretic conditions for the existence of the action of a group $G$ of some order $|G| = N$ on a closed Riemann surface $X$ of genus $\sigma$. These two conditions are: 
\begin{enumerate}
\item $m_i|N$ for $1\leq i\leq r$, and 
\item the satisfaction of the Riemann-Hurwitz formula:
\begin{eqnarray}
\sigma -1 = N \left( h - 1+ \frac{1}{2} \sum_{i=1}^r \left( 1-\frac{1}{m_i} \right) \right).
\label{riemann hurwitz}
\end{eqnarray}
\end{enumerate}
\label{definition potential signature}
\end{definition}

Denote the set of all potential signatures for a fixed $\sigma\ge 2$ by $\mathcal{P}_{\sigma}$.

We have deliberately chosen a narrow definition of potential signature, using only low level arithmetic information.  This choice follows from the statement of Theorem \ref{theorem existence} below.

It is straightforward to see that $\mathcal{P}_{\sigma}$ is finite for every $\sigma\ge 2$.  We start by noting that for a given $\sigma\ge 2$ and a given potential signature $(h; m_1,\ldots, m_r)$ for genus $\sigma$, the values of $h$ and $r$ satisfy $r + 4h \le 2\sigma + 2$, see Lemma 3.1 of Anderson and Wootton \cite{anderson wootton}, and so $\sigma$ bounds both $h$ and $r$.  We have the Hurwitz bound that $N\le 84 (\sigma -1)$, and condition $1.$ of Definition \ref{definition potential signature} then bounds the $m_i$ by $N$.

Given a potential signature $s = (h; m_1,\ldots, m_r)$ and a genus $\sigma \ge 2$, a standard question is to determine conditions under which this potential signature $s$ is in fact the signature of the action of some group $G$ of conformal homeomorphisms on some closed Riemann surface $X$ of genus $\sigma$.  

One such condition involves a {\em generating vector} for $G$ associated to that signature.

\begin{definition} For a finite group $G$, a vector $(a_{1}, b_{1}, a_{2}, b_{2},\dots , a_h, b_h, c_{1},\dots , c_{r})$ of elements of $G$ is an {\em $(h;m_{1},\dots ,m_{r})$-generating vector for $G$} if the following three conditions hold:

\begin{enumerate}
\item $G=\langle a_{1},b_{1},a_{2},b_{2},\dots ,a_{h},b_{h},c_{1},\dots ,c_{r} \rangle$;
\item The order of $c_{i}$ is $m_{i}$ for $1\leq i\leq r$;
\item $\prod_{i=1}^{h} [a_{i} ,b_{i} ] \prod_{j=1}^{r} c_{j}$=1.
\end{enumerate}
\label{definition generating vector}
\end{definition}

The following adapted version of Riemann's existence theorem can be found in Broughton \cite{Bro}.

\begin{theorem} Let $X$ be a closed Riemann surface of genus $\sigma\ge 2$.  There exists an action on $X$ by a group $G$ of conformal homeomorphisms of order $|G| = N$ with signature $(h; m_1,\ldots, m_r)$ if and only if both $(h; m_1,\ldots, m_r)$ is a potential signature for $\sigma$, so that $(h; m_1,\ldots, m_r)\in \mathcal{P}_{\sigma}$, and there exists an $(h;m_{1},\dots ,m_{r})$-generating vector for $G$.
\label{theorem existence}
\end{theorem}

If these conditions are met, we call $(h; m_1,\ldots, m_r)$ an {\em actual signature} for $\sigma$ and denote the set of all actual signatures for genus $\sigma\ge 2$ by $\mathcal{A}_{\sigma}$. 

Let  ${\mathfrak p} = \cap_{\sigma\ge 2} {\mathcal P}_\sigma$ and  ${\mathfrak a} = \cap_{\sigma\ge 2} {\mathcal A}_\sigma$ denote the set of {\em omnipersistent potential signatures} and {\em omnipersistent actual signatures}, respectively, which are then the potential signatures and signatures, respectively, that occur for all genera $\sigma\ge 2$.

It is clear that ${\mathcal A}_\sigma\subset {\mathcal P}_\sigma$ for every $\sigma\ge 2$, though we do not yet understand how ${\mathcal A}_\sigma$ sits inside ${\mathcal P}_\sigma$ for a given value of $\sigma$. Indeed, this is the very question at the center of our proposed approach to group-signature classification: to first understand potential signature space, and then analyze how the actual signature space sits inside the space of potential signatures.

\section{The Lattice Structure on Potential Signatures}
\label{section lattice}

The purpose of this Section is to demonstrate that the collection ${\mathcal P} = \{ {\mathcal P}_\sigma\: |\: \sigma\ge 2\}$ of the sets of potential signatures, partially ordered by the divisibility of $\sigma -1$, forms a lattice.   

Recall that a partially ordered set $(L, \le)$ forms a lattice if each $2$-element subset $\{ a, b\}\subset L$ of $L$ has a greatest lower bound, the {\em join} $a\wedge b$, and a least upper bound, the {\em meet} $a\vee b$, with the usual definitions of greatest lower bound and least upper bound. 

We start by examining how the sets ${\mathcal P}_\sigma$ of potential signatures relate to one another for different values of $\sigma$, or more precisely for different values of $\sigma -1$.  The following Proposition establishes the partial ordering on ${\mathcal P}$ with respect to $\sigma -1$.

\begin{proposition} $\mathcal{P}_{\sigma} \subseteq \mathcal{P}_{\sigma '}$ if and only if $(\sigma -1)|(\sigma' -1)$.
\label{prop containment}
\end{proposition}

\begin{proof} First suppose that $(\sigma -1)|(\sigma' -1)$ and let $(h;m_1,\dots ,m_r)\in \mathcal{P}_{\sigma}$.  Definition \ref{definition potential signature} is then satisfied, and in particular the Riemann-Hurwitz formula is  satisfied, and so there is some $N$ for which 
$$\sigma -1=N \left( h-1 +\frac{1}{2} \sum_{i=1}^{r} \left( 1-\frac{1}{m_i} \right) \right).$$ 
Set $K=(\sigma' -1)/(\sigma -1)$.  Multiplying both sides of the equation by $K$, we get 
$$\sigma' -1=NK \left( h-1 +\frac{1}{2} \sum_{i=1}^{r} \left( 1-\frac{1}{m_i} \right) \right).$$ 
Moreover, since $m_j | N$ for $1\le j\le r$, we have that $m_j | NK$ for $1\le j\le r$.  

Since both conditions of Definition \ref{definition potential signature} are satisfied, we see that the tuple $(h;m_1,\dots ,m_r)\in {\mathcal P}_{\sigma'}$ is a potential signature for the action by conformal homeomorphisms of some group $G'$ of order $NK$ on a closed Riemann surface of genus $\sigma'$, and so ${\mathcal P}_{\sigma}\subseteq {\mathcal P}_{\sigma'}$.  As a note, we remind ourselves that we don't know that such a group $G'$ actually exists or what it might be if it does exist, as we are working with potential signatures.

Now suppose that  $(\sigma -1)\nmid (\sigma' -1)$. To show that $\mathcal{P}_{\sigma} \not\subseteq \mathcal{P}_{\sigma '}$ we give an example of a potential signature in $\mathcal{P}_{\sigma}$ that is not in $\mathcal{P}_{\sigma '}$. Consider the signature $(0;2,2\sigma+1,4\sigma+2)$. 

It is known that $(0;2,2\sigma+1,4\sigma+2)$ is an actual signature for the action by conformal homeomorphisms of some group $G_\sigma$ of order $4\sigma + 2$ on a closed Riemann surface of genus $\sigma$ for every $\sigma\ge 2$, see for Example 9.7 of Breuer \cite{Breu}.  As an actual signature, we then immediately have that $(0;2,2\sigma+1,4\sigma+2)\in {\mathcal P}_{\sigma}$ for every $\sigma\ge 2$.

Suppose $(0;2,2\sigma+1,4\sigma+2)$ is also a potential signature for the action by conformal homeomorphisms of some group $G$ of order $N$ on a closed Riemann surface of genus $\sigma'\ne\sigma$.  In this case, we see using Definition \ref{definition potential signature} that $(0;2,2\sigma+1,4\sigma+2)$ would then satisfy the Riemann-Hurwitz formula, which simplifies to
$$\sigma' -1=\frac{N}{2} \left( \frac{\sigma -1}{2\sigma+1} \right),$$ 
and so
$$N=\frac{2(2\sigma+1)(\sigma' -1)}{\sigma -1} .$$ 
Now since $(0;2,2\sigma+1,4\sigma+2)$ contains a period equal to $4\sigma+2$, it must be that $4\sigma+2$ then divides $N$, so $N=(4\sigma+2)K$ for some integer $K$, noting here that $K>1$ as otherwise $\sigma'=\sigma$. 

Therefore, we get 
$$N=(4\sigma+2)K=\frac{2(2\sigma+1)(\sigma' -1)}{\sigma -1},  \text{ giving } K=\frac{\sigma' -1}{\sigma -1}.$$ 
In particular, $(\sigma' -1)/(\sigma -1)$ is an integer and so $(\sigma -1)|(\sigma' -1)$, a contradiction. 

Therefore, if $(\sigma -1)\nmid (\sigma' -1)$, then $(0;2,2\sigma+1,4\sigma+1) \in \mathcal{P}_{\sigma}$ but $(0;2,2\sigma+1,4\sigma+1) \not\in \mathcal{P}_{\sigma '}$ and therefore $\mathcal{P}_{\sigma} \not\subseteq \mathcal{P}_{\sigma '}$.
\end{proof}

Proposition \ref{prop containment} allows us to define the meet ${\mathcal P}_\sigma\wedge {\mathcal P}_{\sigma'}$ as the intersection $\mathcal{P}_{\sigma} \cap \mathcal{P}_{\sigma '}$.

\begin{corollary} (The meet of ${\mathcal P}_\sigma$ and ${\mathcal P}_{\sigma'}$) If $\gcd{\left( (\sigma -1), (\sigma' -1) \right)}=\Sigma -1$, then $\mathcal{P}_{\sigma} \cap \mathcal{P}_{\sigma '} =\mathcal{P}_{\Sigma}$.
\label{corollary intersection}
\end{corollary}

\begin{proof} By Proposition \ref{prop containment}, we know that $\mathcal{P}_{\sigma} \cap \mathcal{P}_{\sigma '} \supseteq \mathcal{P}_{\Sigma}$, since $\Sigma-1$ divides both $\sigma -1$ and $\sigma' -1$. 

Therefore, we just need to show containment in the other direction, so suppose that $(h;m_1,\dots ,m_r) \in \mathcal{P}_{\sigma} \cap \mathcal{P}_{\sigma '}$. Then the Riemann-Hurwitz formula is satisfied for both $\sigma$ and $\sigma'$, and so there exist $K$ and $N$ so that
$$\sigma -1=N \left( h-1 +\frac{1}{2} \sum_{i=1}^{r} \left( 1-\frac{1}{m_i} \right) \right)$$ 
and 
$$\sigma' -1=K \left( h-1 +\frac{1}{2} \sum_{i=1}^{r} \left( 1-\frac{1}{m_i} \right) \right).$$ 
Using Bezout's identity, there exists integers $a$ and $b$ such that $\Sigma -1=a(\sigma -1)+b(\sigma' -1)$.  Substituting, we see that
\begin{eqnarray*}
\Sigma -1 & = & aN \left( h-1 +\frac{1}{2} \sum_{i=1}^{r} \left( 1-\frac{1}{m_i} \right) \right) +bK \left( h-1 +\frac{1}{2} \sum_{i=1}^{r} \left( 1-\frac{1}{m_i} \right) \right)  \\
& = & (aN+bK) \left( h-1 +\frac{1}{2} \sum_{i=1}^{r} \left( 1-\frac{1}{m_i} \right) \right).
\end{eqnarray*}

Moreover, we have that $m_j | N$ and $m_j | K$ for $1\le j\le r$, and so clearly we have that $m_j | (aN + bK)$ for $1\le j\le r$.

Combining the above, we see by Definition \ref{definition potential signature} that $(h;m_1,\dots ,m_r) \in \mathcal{P}_{\Sigma}$, as desired.
\end{proof}

Proposition \ref{prop containment} also allows us to determine the join ${\mathcal P}_\sigma\vee {\mathcal P}_{\sigma'}$ in terms of the least common multiple of $\sigma -1$ and $\sigma' -1$.

\begin{corollary} (The join of ${\mathcal P}_\sigma$ and ${\mathcal P}_{\sigma'}$) The smallest value of $\Sigma$ for which $\mathcal{P}_{\Sigma}$ contains both $\mathcal{P}_{\sigma}$ and $\mathcal{P}_{\sigma '}$ satisfies ${\rm lcm}{\left( (\sigma -1), (\sigma' -1) \right)}=\Sigma -1$. Moreover, if $\mathcal{P}_{\Sigma'}$ is any other set containing $\mathcal{P}_{\sigma}$ and $\mathcal{P}_{\sigma '}$, then $(\Sigma -1)|(\Sigma'-1)$.
\end{corollary}

\begin{proof} By Proposition \ref{prop containment}, $\mathcal{P}_{\Sigma}$ contains both $\mathcal{P}_{\sigma}$ and $\mathcal{P}_{\sigma '}$ if and only if $(\Sigma -1)|(\sigma -1)$ and $(\Sigma -1)|(\sigma' -1)$, so in particular, if and only if $\Sigma-1$ is a common multiple of $\sigma -1$ and $\sigma' -1$. The result follows.
\end{proof}

\begin{corollary} The collection ${\mathfrak P} = \{ {\mathcal P}_\sigma\: |\: \sigma\ge 2\}$ of the sets of potential signatures, partially ordered by the divisibility of $\sigma -1$, forms a lattice.   
\label{cor-lat}
\end{corollary}

As remarked in Section \ref{introduction}, the growth rate of group actions between genera is not monotone -- it appears to jump around somewhat haphazardly. However, the lattice structure given in Corollary \ref{cor-lat} does provide some partial explanation for this. Specifically, from a naive point of view, more potential signatures should signal more group actions, so in particular, the larger the number of divisors of $\sigma -1$, the more group actions we would expect. This observation is supported in the available data as illustrated, see for example Table \ref{obs-dat} where we provide the number of distinct actions (up to signature) in a given genus. 

\begin{table}[ht]
\begin{center}
\begin{tabular}{|c|c||c|c||c|c|}
\hline
Genus & \# Actions & Genus & \# Actions & Genus & \# Actions \\
\hline
2 & 21 & 3 & 49 & 4 &  64 \\ \hline
 5 &  93 & 6 & 87 & 7 &  148 \\ \hline
8 & 108  & 9 & 268 & 10 & 226 \\ \hline
 11 & 232 & 12 & 201 & 13 & 453 \\ \hline
\hline
\end{tabular}
\end{center}
\caption{\label{obs-dat}Actual Signature Space Sizes for Low Genus}
\end{table}

Interestingly, the lattice structure for $\{ {\mathcal P}_\sigma\: |\: \sigma\ge 2\}$ with respect to the divisibility of $\sigma -1$ does not descend to a lattice structure for $\{ {\mathcal A}_\sigma\: |\: \sigma\ge 2\}$, as we will see below. 

\begin{proposition} The collection ${\mathfrak A}$ of actual signatures does not admit a lattice structure arising from the divisibility of $\sigma -1$.
\label{prop actual no lattice}
\end{proposition}

\begin{proof}  If the actual signatures ${\mathfrak A} = \{{\mathcal A}_\sigma\: |\: \sigma\ge 2\}$ were partially ordered by the divisibility of $\sigma -1$, then we would have that ${\mathcal P}_3\subset {\mathcal P}_n$ for all $n$ odd.  

However, we have famously that $(0; 2, 3, 7)\in {\mathcal A}_3$ but we also have $(0; 2, 3, 7)\not\in {\mathcal A}_5$, see for example the database from \cite{Breu}.
\end{proof}

We close this Section with the following conjecture.  

\begin{conjecture} There is no lattice structure on $\{ {\mathcal A}_\sigma\: |\: \sigma\ge 2\}$.
\end{conjecture}

\section{Omnipersistent actual signatures}
\label{section omnipersistent}

The purpose of this section is to prove Theorems \ref{theorem potential omnipersistent} and \ref{theorem omnipersistent}. Theorem \ref{theorem potential omnipersistent} is a simple consequence of the lattice structure of ${\mathcal P}$.

\begin{theorem} The omnipersistent potential signatures are 
\[ {\mathfrak p} = \cap_{\sigma\ge 2} {\mathcal P}_\sigma = {\mathcal P}_2. \]
\label{theorem potential omnipersistent}
\end{theorem}

\begin{proof} We know by definition that ${\mathfrak p} = \cap_{\sigma\ge 2} {\mathcal P}_\sigma$; by Corollary \ref{cor-lat}, we see that $\cap_{\sigma\ge 2} {\mathcal P}_\sigma ={\mathcal P}_2$.
\end{proof}

It is a straightforward though somewhat tedious calculation, as described after Definition \ref{definition potential signature}, to show that ${\mathcal P}_2$ contains the (potential) signatures
\[ \left\{ \begin{array}{ccccc} 
( 0; 2, 2, 2, 2, 2, 2) & ( 0; 2, 6, 6 ) & ( 0; 2, 2, 4, 4 ) & ( 0; 3, 6, 6 ) & ( 0; 5, 5, 5 ) \\
( 0; 2, 2, 2, 2, 2 ) & ( 0; 2, 2, 3, 3 ) & ( 0; 2, 2, 2, 6 ) & ( 0; 4, 4, 4 ) & ( 0; 2, 8, 8 ) \\
( 0; 2, 2, 2, 4 ) & ( 0; 3, 3, 9 ) & ( 0; 2, 5, 10 ) & ( 0; 3, 4, 4 ) & ( 0; 3, 3, 6 )\\
( 0; 3, 3, 3, 3 )  & ( 0; 2, 4, 12 ) & ( 0; 2, 2, 2, 3 ) & ( 0; 3, 3, 5 ) & ( 0; 2, 4, 8 ) \\
 ( 0; 2, 3, 18 ) & ( 0; 2, 5, 5 ) & ( 0; 3, 3, 4 ) & ( 0; 2, 4, 6 ) & ( 0; 2, 3, 12 ) \\
 ( 0; 2, 3, 10 ) & ( 0; 2, 3, 9 )  & ( 0; 2, 4, 5 )  & ( 0; 2, 3, 8 ) & ( 0; 2, 3, 7 ) \\
 ( 1; 2, 2 ) & ( 1; 3 ) & ( 1; 2 )  & (2;-) &  
\end{array}   \right\}\]

As remarked in Section \ref{introduction}, there are many tuples that frequently arise as actual signatures significantly more so than others, and Theorem \ref{theorem potential omnipersistent} provides some explanation for this. We previously remarked that $(0;2,3,7)$ regularly shows up, but perusing through lists for low genus, we see for example $( 0; 2, 2, 2, 2, 2, 2)$ appearing in every genus, and $(0;4,4,4)$ appearing in nearly every genus. 

The explanation for this of course is that the potential signatures in $\mathcal{P}_{2}$ are the only potential signatures, and hence actual signatures, that might possibly appear in all but finitely many genera, and so the potential signatures in ${\mathcal P}_2$ are the only actual signatures we would expect to see with a high degree of frequency.

In fact, we can say a bit more.  Let $X\subset {\mathbb N}\setminus \{ 1\}$ be any collection, finite or infinite, that contains consecutive integers $n\ge 2$ and $n+1$.  We then have that
\[ {\mathcal P}_2 = \cap_{\sigma\ge 2} {\mathcal P}_\sigma\subseteq \cap_{\sigma\in X} {\mathcal P}_\sigma \subseteq {\mathcal P}_n\cap {\mathcal P}_{n+1} = {\mathcal P}_2, \]
where the first equality follows from Theorem \ref{theorem potential omnipersistent}; the first inclusion follows from the assumption that $X\subset {\mathbb N}\setminus \{ 1\}$; the second inclusion follows from the fact that $\{ n, n+1\}\subset X$; and the final equality follows from Corollary \ref{corollary intersection} since ${\rm gcd} (n-1, n) = 1 = 2-1$ for $n\ge 2$.

The strategy of the proof for Theorem \ref{theorem omnipersistent} is also straightforward.  We first consider the set ${\mathcal A}_2$ of those actual signatures for genus $\sigma = 2$, which is a well established set, and remove from this list those signatures which do not appear as a signature for some genus $3\le \sigma\le 48$, using the lists generated by Breuer.  

For the four signatures that remain from this process, we then construct generating vectors and apply Theorem \ref{theorem existence}, to show that they are indeed actual signatures for all genera $\sigma\ge 2$.

\begin{theorem}
The omnipersistent actual signatures are 
\[ {\mathfrak a} = \{ (2; -), (1;2,2), (0;2,2,2,2,2), (0;2,2,2,2,2,2)\}. \]
\label{theorem omnipersistent}
\end{theorem}

\begin{proof} The actual signatures for genus $\sigma = 2$ are given in Table \ref{tab-actual-sigs}, along with the smallest genus for which the signature fails to appear for that value of $\sigma$ in the column Smallest Genus.  The four signatures with NA listed as the smallest genus are those which occur as signatures for all $2\le \sigma\le 48$.  Here, we used the lists generated by Breuer \cite{Breu}.

\begin{table}[ht]
\begin{tabular}{||c|c|c||c|c|c||}
\hline
Signature & Order & Smallest  & Signature & Order &  Smallest   \\
 &  &  Genus &  &  &   Genus  \\
\hline\hline
$(2,-)$ & $\sigma -1$ & NA & $(1;2,2)$ & $2(\sigma -1)$ & NA  \\
\hline
$(0; 2, 3, 8)$ & $48(\sigma -1)$ & 4 & $(0; 2, 4, 6)$ & $24(\sigma -1)$ & 7 \\
\hline
$(0; 3, 3, 4)$ & $24(\sigma -1)$ & 4 & $(0; 2, 4, 8)$ & $16(\sigma -1)$ & 4 \\
 \hline
$(0; 2, 2, 2, 3)$ & $12(\sigma -1)$ & 7  & $(0; 2, 6, 6)$ & $12(\sigma -1)$ & 6\\
\hline
$(0; 3, 4, 4)$ & $12(\sigma -1)$ & 6 & $(0; 2, 5, 10)$ & $10(\sigma -1)$ & 3 \\
\hline
  $(0; 2, 2, 2, 4)$ & $8(\sigma -1)$ & 6 & $(0; 2, 8, 8)$ & $8(\sigma -1)$  & 4  \\
\hline
$(0;4,4,4)$ & $8(\sigma -1)$  & 6 & $(0; 2, 2, 3, 3)$ & $6(\sigma -1)$ & 6\\
\hline
$(0; 3, 6, 6)$ & $6(\sigma -1)$ &  6 & $(0; 5, 5, 5)$ & $5(\sigma -1)$ & 3   \\
\hline
$(0; 2, 2, 2, 2, 2)$ & $4(\sigma -1)$ & NA  & $(0; 2, 2, 4, 4)$ & $4(\sigma -1)$ & 4  \\
\hline
 $(0; 3, 3, 3, 3)$ & $3(\sigma -1)$ & 3 & $(0; 2, 2, 2, 2, 2, 2)$ & $2(\sigma -1)$ & NA \\
\hline
\end{tabular}
\caption{\label{tab-actual-sigs} Actual Sigatures in Genus $\sigma$}
\end{table}

By inspection, we are left with $(2;-)$, $(1;2,2)$, $(0;2,2,2,2,2,2)$ and $(0;2,2,2,2,2)$, so we proceed by cases to show that each of these signatures arises as the signature associated to the action of a group of conformal homeomorphisms on a closed Riemann surface of genus $\sigma$ for all $\sigma\geq 2$. 

First consider the signature $(2;-)$. Let $C_{\sigma -1}=\langle x\rangle$ denote the cyclic group of order $\sigma -1$ with generator $x$. Then $(x,e,x,e)$ is a $(2;-)$-generating vector for $C_{\sigma -1}$ for each possible $\sigma\geq 2$. It follows that $C_{\sigma -1}$ acts on a closed Riemann surface of genus $\sigma$ with signature $(2;-)$ for all $\sigma\geq 2$. Hence $(2;-)$ is an omnipersistent actual signature. 

In this case, there is a clear geometric description of the action of $C_{\sigma -1}$ on a particular Riemann surface $X_\sigma$ of genus $\sigma$.  Namely, view $X_\sigma$ as a torus with $\sigma -1$ equally spaced and symmetric handles, so that there is a natural rotation of $X_\sigma$ of order $\sigma -1$.

For the signatures $(1;2,2)$ and $(0;2,2,2,2,2,2)$, denote the dihedral group of order $2(\sigma -1)$ by $D_{\sigma -1}=\langle x,y|x^{\sigma -1},y^2,yxyx\rangle$. Then $(x,e,y,y)$ is a $(1;2,2)$-generating vector for $D_{\sigma -1}$ and $(y,y,xy,xy,y,y)$ is a $(0;2,2,2,2,2,2)$-generating vector for $D_{\sigma -1}$, both for all $\sigma\geq 2$. It follows that the group $D_{\sigma -1}$ acts on closed Riemann surfaces of genus $\sigma$  with signatures $(1;2,2)$ and $(0;2,2,2,2,2,2)$ for all $\sigma\geq 2$. Hence $(1;2,2)$ and $(0;2,2,2,2,2,2)$ are omnipersistent actual signatures. 

Finally we need to consider the signature $(0;2,2,2,2,2)$. Similar to the above, we let $D_{2(\sigma -1)}=\langle x,y|x^{2(\sigma -1)},y^2,yxyx\rangle$ denote the dihedral group of order $4(\sigma -1)$. The vector $(xy,xy,y,yx^{\sigma -1},x^{\sigma -1})$ is then a $(0;2,2,2,2,2)$-generating vector for $D_{2(\sigma -1)}$ for each possible $g\geq 2$. It follows that $D_{2(\sigma -1)}$ acts on a closed Riemann surface of genus $\sigma$ with signature $(0;2,2,2,2,2)$ for all $\sigma\geq 2$. Hence $(0;2,2,2,2,2)$ is an omnipersistent actual signature. 

\end{proof}

\end{document}